\documentclass{amsart}

\usepackage{amsmath}
\numberwithin{equation}{section}

\usepackage{amsthm}
\usepackage{hyperref}
\usepackage{enumerate}
\usepackage{MnSymbol}
\usepackage{amsthm}

\makeatletter
\newtheorem*{rep@theorem}{\rep@title}
\newcommand{\newreptheorem}[2]{%
\newenvironment{rep#1}[1]{%
 \def\rep@title{#2 \ref{##1}}%
 \begin{rep@theorem}}%
 {\end{rep@theorem}}}
\makeatother

\newtheorem{theorem}{Theorem}[section]
\newtheorem{maintheorem}{Theorem}

\newreptheorem{maintheorem}{Theorem}

\newtheorem{proposition}[theorem]{Proposition}

\newtheorem{corollary}[theorem]{Corollary}
\newtheorem{lemma}[theorem]{Lemma}
\theoremstyle{definition}

\newtheorem{remark}[theorem]{Remark}
\newtheorem{example}[theorem]{Example}

\newcommand{\defnword}[1]{\textbf{#1}}

\makeatletter
\g@addto@macro\bfseries{\boldmath}
\makeatother

\title{Smooth manifolds with prescribed rational cohomology ring}

\author{Jim Fowler}
\address[J.~Fowler]{The Ohio State University, Columbus, Ohio 43210}
\email[J.~Fowler]{fowler@math.osu.edu}

\author{Zhixu Su}
\address[Z.~Su]{University of California, Irvine, California 92697}
\email[Z.~Su]{zhixus@uci.edu}

\newcommand{\Z}{\mathbb{Z}}

\newcommand{\Q}{\mathbb{Q}}

\newcommand{\OP}{\mathbb{O}P}
\newcommand{\CP}{\mathbb{C}P}
\newcommand{\HP}{\mathbb{H}P}

\newcommand{\LP}{{\mathcal{L}}}
\newcommand{\A}{{\mathcal{A}}}
\newcommand{\df}{\displaystyle\frac}
\newcommand{\ds}{\displaystyle\sum}
\newcommand{\dpr}{\displaystyle\prod}
\DeclareMathOperator{\numer}{numer}
\DeclareMathOperator{\denom}{denom}

\DeclareMathOperator{\wt}{wt}

\usepackage{todonotes}

\begin{document}
\begin{abstract}
  The Hirzebruch signature formula provides an obstruction to the
  following realization question: given a rational Poincar\'e duality
  algebra $\A$, does there exist a smooth manifold $M$ such that
  $H^*(M;\Q)=\A$?

  This problem is especially interesting for rational truncated
  polynomial algebras whose corresponding integral algebra is not
  realizable.  For example, there are number theoretic constraints on
  the dimension $n$ in which there exists a closed smooth
  manifold $M^n$ with $H^*(M^n;\Q)= \Q[x]/\langle x^3\rangle$.  We limit the possible existence dimension to $n=8(2^a+2^b)$. For $n
  = 32$, such manifolds are not two-connected. We show that the next smallest
  possible existence dimension is $n=128$. As there exists no integral $\OP^m$ for $m>2$, the realization of the truncated polynomial algebra $\Q[x]/\langle x^{m+1}\rangle, |x|=8$ is studied. Similar considerations provide
  examples of topological manifolds which do not have the rational
  homotopy type of a smooth closed manifold.

  The appendix presents a recursive algorithm for efficiently
  computing the coefficients of the $\LP$-polynomials, which arise in
  the signature formula.
\end{abstract}

\maketitle

\section{Introduction}
\label{section:introduction}

Let $\A$ be a $1$-connected graded commutative algebra over $\Q$
satisfying the Poincare duality. We ask if there exists a
simply-connected closed manifold $M$ such that $H^*(M;\Q)=\A$. In
general, such realization problem can be studied by rational surgery
(\cite{MR440574}, \cite{MR646078}). When the dimension of $\A$ is $n=4k$, there
exists such a manifold $M$ only if there is a choice of fundamental
class $\mu\in (\A^0)^*$ such that the bilinear form on $\A^{2k}$
defined as $(\cdot \smile \cdot, \mu)$ has the signature
$\sigma(\A,\mu)$ equal to the signature of the manifold $\sigma(M)$,
which by Hirzebruch signature theorem is equal to the $\LP$-genus
$\langle\LP_k(p(\tau_M), [M]\rangle$.

The case of $\A=\Q[x]/\langle x^3\rangle, |x|=2k, k>8$ was studied in
\cite{MR3158765}. A closed smooth manifold with such rational
cohomology ring is called a \defnword{rational projective plane}. It
is shown that a rational projective plane must be of dimension $n=8k$
for $k>1$.  After 4, 8, and 16 (i.e., the dimensions of $\CP^2$,
$\HP^2$ and $\OP^2$ respectively), the next smallest $n$ for which
there is an $n$-dimensional rational projective plane is $n = 32$.

\subsection*{Main results}

In Section~\ref{section:rational-planes}, by a number theoretic
observation on the coefficients of the signature equation, we limit
the possible existence dimensions of a projective plane to
$n=8(2^a+2^b)$, where $a, b$ are arbitrary nonnegative
integers. 
\begin{maintheorem}\label{maintheorem:two-nonzero-bits}
  For $n\geq 8$, a rational projective plane can only exist in
  dimension $n = 8k$ with $k=2^a+2^b$ for some nonnegative integers
  $a, b$.
\end{maintheorem}
Certain specific dimensions after 32 will be eliminated in
Section~\ref{section:very-high-dimensional}, so that the next smallest
$n$ for which an $n$-dimensional rational projective plane could exist
is $n=128$.

In Section~\ref{section:connectivity}, we study the connectivity of
the 32-dimensional rational projective planes from
\cite{MR3158765}. In particular, we show that such a manifold does not
admit a Spin structure and thus is not 2-connected.  As a consequence,
we have the following result.
\begin{maintheorem}\label{maintheorem:not-two-connected}
  A 32-dimensional rational projective plane is not 2-connected.
\end{maintheorem}

The truncated polynomial algebra $\Q[x]/\langle x^{m+1}\rangle, |x|=8$
will be considered in Section~\ref{section:octonionic-examples}. While the construction
of $\OP^2$ does not generalize to give any higher dimensional
octonionic projective spaces, we show that a rational $\OP^{m}$ exists for $m$ odd.
\begin{maintheorem}\label{maintheorem:rational-op3}
 If $m>2$ is odd, there exists a closed smooth $8m$-dimensional manifold $M$ with rational cohomology ring
  $$
  H^*(M;\Q) \cong \Q[x]/\langle x^{m+1}\rangle \quad |x| = 8.
  $$
\end{maintheorem}

Finally, in Section~\ref{section:e8-manifolds}, we use similar
approach to study if a Milnor's $E_8$ manifold have the rational
homotopy type of a smooth manifold.
\begin{maintheorem}\label{maintheorem:e8-manifolds}
  If an $E_8$ manifold $M^{4k}$ $(k>1)$ has the rational homotopy type
  of a smooth manifold, then $k$ is even and $k$ has no more than 5
  nonzero bits in its binary expansion. In particular, $M^8$ has the
  rational homotopy type of a smooth manifold.
\end{maintheorem}
A surprising consequence is that the $504$-dimensional $E_8$ manifold
is a simply-connected topological manifold which does not have the
rational homotopy type of a smooth manifold.

In the Appendix, we give an \textit{efficient} recursive algorithm for
computing the coefficients of the $\LP$-polynomial.  Such an algorithm
is a useful tool for the study of the signature obstruction.

\section{Rational projective planes}
\label{section:rational-planes}

The Hopf Invariant One theorem implies that $\CP^2, \HP^2, \OP^2$ are
the only closed manifolds whose integral cohomology ring is
$H^*(-;\Z)\cong\Z[x]/\langle x^3\rangle$. Ignoring the torsion, we
seek realization of the rational cohomology algebra by closed smooth
manifolds in dimension greater than 16. We call a simply-connected
smooth closed manifold $M^{4k}$ such that $H^*(M;\Q)\cong\Q[x]/\langle
x^3\rangle$ a rational projective plane. It has been shown that a
rational projective plane only exists in dimension $n=8k$ if $k>1$
(\cite{MR3158765}).

If $M^{8k}$ is a rational projective plane, by the Hizebruch signature theorem, 
$$\langle\LP_k(p(\tau_M)),[M]\rangle=\sigma(M)=\pm1$$
Since $H^*(M;\Q)\cong\Q[x]/\langle x^3\rangle, |x|=4k$, the Pontraygin class $p_i(\tau_M)\in H^{4i}(M;\Q)$ is zero for all $i$ except $p_{k}(\tau_M)$ and $p_{2k}(\tau_M)$, the signature equation then becomes
$$s_{k,k}\langle p_{k}^2(\tau_M),[M]\rangle+s_{2k}\langle p_{2k}(\tau_M),[M]\rangle=\pm 1$$
Let $\alpha\in H^{4k}(M;\Z)$ be the generator such that $\langle\alpha\smile\alpha, [M]\rangle=\pm1$, then $p_{k}(\tau_M)=x\alpha$ and $p_{2k}(\tau_M)=y\alpha^2$ for some integers $x$ and 
$y$. Then we have the diophantine equation
\begin{equation} \label{eq}
  s_{k,k}x^2+s_{2k}y=\pm 1.
\end{equation}
A formulas for $s_{k,k}$ and $s_{2k}$ can be found below in Equation~\eqref{sk} and Equation~\eqref{sk2}.

By examining the coefficients $s_{k,k}$ and $s_{2k}$, we will prove the following theorem.
\begin{repmaintheorem}{maintheorem:two-nonzero-bits}
For $n\geq 8$, a rational projective plane can only exist in dimension $n=8k$ with $k=2^a + 2^b$ for some 
  nonnegative integers $a, b$.
\end{repmaintheorem}
\begin{proof}
We firstly prove the following proposition on the $2$-adic valuation of the coefficients $s_{k,k}$ and $s_k$.

\begin{proposition}
  \label{proposition:2-adic}
  Let $k$ be an integer with $m>2$ nonzero bits in its binary
  expansion.  Then the numerators of the irreducible fractions of $s_{k,k}$ and $s_{2k}$ are
  divisible by $2^{m-2}$.
\end{proposition}

\begin{proof}
We fix our notation: the {\it Hamming weight} function $\wt: \Z
\to \Z$ sends an integer $x$ to the number of 1's in the binary
expansion of $x$; the {\it$2$-adic valuation,} $\nu_2 : \Z \to
\Z$, sends $x$ to the largest integer $\nu$ so that $2^{\nu}$ divides
$x$; the extended {\it$2$-adic valuation,} $\nu_2 : \Q \to
\Z$ is defined as $\nu_2(\frac{x}{y})=\nu_2(x)-\nu_2(y)$. Then the proposition asserts that $\nu_2(s_{k,k})$ and
$\nu_2(s_{2k})$ are at least $\wt(k) - 2$ when $\wt(k)\geq 2$. We will need the following lemma to prove the proposition.

\begin{lemma}
  For any integer $x$,
  \[
  \wt(x) + \nu_2(x!) = x.
  \]
\end{lemma}
\begin{proof}
  Let $\lfloor - \rfloor$ and $\{ - \}$ denote the floor and
  fractional part function, respectively. Then
  \begin{align*}
    x
&= \sum_{k=1}^\infty \frac{x}{2^k} \\
&= \left( \sum_{k=1}^\infty \left\{ \frac{x}{2^k} \right\} \right) + \left( \sum_{k=1}^\infty \left\lfloor\frac{x}{2^k}\right\rfloor \right) \\
&= \wt(x) + \nu_2(x!).
  \end{align*}
\end{proof}

Let $B_{k}=\df{\numer(B_{k})}{\denom(B_{k})}$ denotes the irreducible
fraction. It is a fact that
\[
\denom(B_{k})=\displaystyle\prod_{p-1\,|\, k}p
\] where the product is over all the primes $p$ such that $p-1\,|\,
k$. This implies that when $k>0$, $\nu_2(B_{k}) = -1$.

By formula \eqref{sk}, the coefficient
$$s_{k}=\displaystyle\frac{2^{2k}(2^{2k-1}-1)|B_{2k}|}{(2k)!},$$
then the 2-adic valuation of $s_{k}$ is 
\begin{align*}
\nu_2(s_{k}) 
&= \nu_2(2^{2k}) + \nu_2(2^{2k - 1} - 1) + \nu_2(B_{2k}) - \nu_2\left((2k)!\right) \\
&= 2k + 0 + \nu_2(B_{2k}) - \nu_2\left((2k)!\right) \\
&=(\wt(2k)+\nu_2\left((2k)!\right)) + \nu_2(B_{k}) - \nu_2\left((2k)!\right) \\
&= \wt(2k) + \nu_2(B_{k}),\\
&=\wt(k)-1.
\end{align*}

We also have $\nu_2(s_{2k})=\wt(2k)-1=\wt(k)-1$. Therefore when
$\wt(k)>1$, $\numer(s_{2k})$ is divisible by $2^{\wt(k)-1}$.

On the other hand, the coefficient
\[
s_{k,k} = \frac{{s_k}^2 - s_{2k}}{2},
\]
so when $k>0$, the 2-adic valuation is 
\begin{align*}
  \nu_2(s_{k,k})
&\geq \min \{ \nu_2({s_k}^2),\nu_2(s_{2k})\} - 1 \\
&= \min \{ 2 \cdot \left(\wt(k) - 1\right)), \wt(k) - 1 \} - 1 \\
&= \wt(k) - 2.
\end{align*}

Therefore when $\wt(k)>2$, $\numer(s_{k,k})$ is divisible by
$2^{\wt(k)-2}$. This proves
Proposition~\ref{proposition:2-adic}.\end{proof}

Now we return to the proof of Theorem~\ref{maintheorem:two-nonzero-bits}. Recall that a
necessary condition for the existence of a rational projective plane
in dimension $8k$ is that the Diophantine equation
\[
s_{k,k}x^2+s_{2k}y=\pm 1
\]
has integer solutions $x$ and $y$. By
Proposition~\ref{proposition:2-adic}, when $\wt(k)>2$, we have $2
\divides \numer(s_{k,k})$ and $2 \divides \numer(s_{2k})$. Therefore
the diophantine equation has no solution when $\wt(k)>2$. Hence a
rational projective plane can only exist in dimension $n=8k$ where
$w(k)=1$ or $2$, i.e., when $k=2^a+2^b$ for some nonnegative integers
$a, b$.

\end{proof}

\section{Ruling out some specific dimensions up to 128}
\label{section:up-to-dimension-128}

It has been shown that after dimension $16$ (for which there is the
octonionic projective plane), the next smallest dimension where a
rational analog of projective plane exists is $32$
(\cite{MR3158765}). By Theorem~\ref{maintheorem:two-nonzero-bits}, the candidate dimensions
are $n=40, 48, 64, 72, 80, 96,128, \ldots$. We will rule out all those
that are less than 128 and show that the signature equation
\textit{does} have a solution for dimension 128.

\begin{proposition}
  There is no rational projective plane in dimension 64.
\end{proposition}
\begin{proof}
   It suffices to show that there is no solution to $s_{8,8} x^2 \pm
  s_{16} y = \pm 1$ for integers $x$ and $y$.  Note that $37$ divides
  the numerator of $s_{16}$, because $37$ divides $B_{32}$, a
  well-known fact considering 37's status as the smallest irregular
  prime.  So it is enough to show there is no solution to $x^2
  \not\equiv \pm 1/s_{8,8} \pmod{37}$.  Since $s_{16} \equiv 0
  \pmod{37}$,
  \begin{align*}
    s_{8,8}
    &\equiv \frac{{s_8}^2 - s_{16}}{2} \equiv \frac{{s_8}^2}{2} \pmod{37},
  \end{align*}
   but neither $2$ nor $-2$ is a quadratic residue modulo 37.
\end{proof}
This method of searching for a particular prime---played by $37$ to
prove nonexistence of a 64-dimensional example---can be applied to
rule out many other possible examples.
\begin{lemma}\label{lemma:rule-out}
  There is no $8n$-dimensional rational projective plane whenever there is a prime $p$ so that
  \begin{itemize}
  \item $2$ and $-2$ are quadratic nonresidues modulo $p$,
  \item $\nu_p(s_{2n}) > 0$, but
  \item $\nu_p(s_{n}) = 0$.
  \end{itemize}
\end{lemma}
To ensure $2$ is a quadratic nonresidue, it is enough that $p
\not\equiv \pm 1 \pmod 8$; to ensure that $-2$ is also a quadratic
nonresidue, we further want $p \equiv 5 \pmod 8$.  By checking the
handful of primes $p$ which divide the numerator of $s_{2n}$, we can
quickly find witnesses which rule out $n$-dimensional projective
planes for $n \in \{48, 64, 72, 96\}$; these witnesses are listed in
Table~\ref{table:witnesses}.

\begin{table}
  \caption{Primes $p$ ruling out $n$-dimensional rational projective planes via Lemma~\ref{lemma:rule-out}.}
  \label{table:witnesses}

  \begin{tabular}{l|l}
    $n$ & $p$ \\
    \hline
    48 & 2294797 \\
    64 & 37 \\
    72 & 26315271553053477373 \\
    96 & 653 \\
    136 & 101 \\
    160 & 10589 \\
  \end{tabular}
\end{table}

Between $32$ and $128$, the only integers $n$ which have at most two
nonzero bits in its binary expansion are $40$, $48$, $64$, $72$, $80$,
and $96$.  So by Proposition~\ref{proposition:2-adic}, these are only
possibilities for the dimensions of a rational projective plane when
$32 < n < 128$.  In light of Table~\ref{table:witnesses}, we can
restrict attention to $n = 40$ and $n = 80$.

\begin{proposition}\label{proposition:no-dimension-40}
  There is no 40-dimensional rational projective plane.
\end{proposition}
\begin{proof}
  If there were, there would be integers $x$ and $y$ for which
\[
  s_{5,5}x^2+s_{10}y=\pm 1,
\]
but there are no such integers.

Suppose there were $x$ and $y$ so that $s_{5,5}x^2+s_{10}y=1$.  We use
the fact that $2^{2\cdot 10-1}-1 = 524287$ is prime and the fact that
$B_{20} = -174611/330$.  Then equation~\eqref{sk} yields
\begin{align*}
s_{10} &= \displaystyle\frac{2^{2\cdot 10}(2^{2\cdot 10-1}-1)|B_{2\cdot 10}|}{(2\cdot 10)!} \\
&= \displaystyle\frac{2^{20} \cdot 524287 \cdot 174611}{330 \cdot (20)!} \\
&= \displaystyle\frac{2^{2\cdot 10} \cdot 524287 \cdot 283 \cdot 617}{330 \cdot (20)!}.
\end{align*}
Set $p = 283$, so $s_{10} \equiv 0 \pmod{283}$.  Then we have a
solution to $x^2 \equiv 1/s_{5,5} \pmod{283}$.  But by Equation~\eqref{sk2},
\begin{align*}
1/s_{5,5} &= -4593988395871875/527062321 \\
&\equiv 146 \pmod{283},
\end{align*}
and $146$ is not a quadratic residue modulo 283, which is a
contradiction.

There is another case to consider: suppose there were $x$ and $y$ so
that $s_{5,5}x^2+s_{10}y=-1$.  In this case, consider $q = 524287$.
Again, we have $s_{10} \equiv 0 \pmod q$, and we assume there is an
integer $x$ so that $x^2 \equiv -1/s_{5,5} \pmod q$.  But arithmetic
reveals the contradiction: $-1/s_{5,5} \equiv 318975 \pmod q$, and
318975 is not a quadratic residue modulo $q$.
\end{proof}
We can boil this proof down into a lemma.
\begin{lemma}\label{lemma:rule-out-using-two-primes}
  There is no $8n$-dimensional rational projective plane whenever there are primes $p$ and $q$ so that
  \begin{itemize}
  \item $2$ and $-2$ are quadratic nonresidues modulo $p$,
  \item $\nu_p(s_{2n}) > 0$ and $\nu_q(s_{2n}) > 0$, 
  \item $1/s_{n,n}$ is a quadratic nonresidue modulo $p$,  and
  \item $-1/s_{n,n}$ is a quadratic nonresidue modulo $q$.
  \end{itemize}
\end{lemma}
We can rule out an 80-dimensional rational projective plane by choosing
\begin{align*}
p &= 1897170067619, \\
q &= 79.
\end{align*}
\begin{proposition}
  There is no $n$-dimensional rational projective plane for $32 < n < 128$.
\end{proposition}

The situation is quite different when $n = 128$.
\begin{proposition}\label{128}
The signature equation (Equation~\eqref{eq}) has a solution in dimension $128$.
\end{proposition}
\begin{proof}
In dimension $128$, the signature equation is
\begin{equation}
  \label{eq128}s_{16,16}x^2+s_{32}y=\pm 1,
\end{equation}
where, using Equations~\eqref{sk} and \eqref{sk2}, we compute 
\begin{align*}
s_{16,16}&=\df{-157 \cdot 311 \cdot 4759 \cdot 7841 \cdot 483239 \cdot 2044352341992497636810897021371}{3^{31} \cdot 5^{16} \cdot 7^{8} \cdot 11^{5} \cdot 13^{3} \cdot 17^{4} \cdot 19^{3} \cdot 23^{2} \cdot 29^{2} \cdot 31^{2} \cdot 37 \cdot 41 \cdot 43 \cdot 47 \cdot 53 \cdot 59 \cdot 61}, \\
s_{32}&=\df{73 \cdot 127 \cdot 337 \cdot 92737 \cdot 649657 \cdot 1226592271 \cdot N}{3^{31} \cdot 5^{15} \cdot 7^{8} \cdot 11^{5} \cdot 13^{4} \cdot 17^{4} \cdot 19^{3} \cdot 23^{2} \cdot 29^{2} \cdot 31^{2} \cdot 37 \cdot 41 \cdot 43 \cdot 47 \cdot 53 \cdot 59 \cdot 61}.
\end{align*}
where $N = 87057315354522179184989699791727$.  Equation~\eqref{eq128} can be simplified to
$$x^2\equiv \pm a\pmod{m}$$
where 
\begin{align*}
  a &= 98719348515711444512355076910350678632922916684640405411745,\\
  m &= 100500713568783890959555031913261799931478908397894537794155 \\
  &= 5 \cdot 73 \cdot 127 \cdot 337 \cdot 92737 \cdot 649657 \cdot 1226592271 \cdot N.
\end{align*}
Using the factorization of $m$, it can be verified that $a$ is a
quadratic residue for each of the factors.  Therefore
Equation~\eqref{128} has a solution.
\end{proof}
\begin{remark}
  The signature equation is only a \textit{necessary} condition for
  the existence of a rational projective plane. By rational surgery
  (\cite{MR440574}, \cite{MR646078}), there exists a 128-dimensional
  rational projective plane if and only if there exists integers
  $x$ and $y$ such that
  \begin{enumerate}[(i)]
  \item The signature equation \eqref{eq128} holds;
  \item\label{condition:congruence} The integers $x^2$ and $y$ are the
    Pontragin numbers of a genuine closed smooth manifold, i.e.,
    $x^2=\langle p_{16}^2(\tau_N), [N]\rangle$ and $y=\langle
    p_{32}(\tau_N), [N]\rangle$ for a 128-dimensional closed smooth
    manifold $N$.
  \end{enumerate}
  Condition~(\ref{condition:congruence}) is equivalent to a set of
  congruence relations on $x^2$ and $y$ (see \cite{MR3158765}), which
  requires a significant amount of calculations in dimension 128.
\end{remark}

\section{Ruling out some very high-dimensional examples}
\label{section:very-high-dimensional}

We apply Lemma~\ref{lemma:rule-out} and
Lemma~\ref{lemma:rule-out-using-two-primes} to rule out some higher
dimensional examples.
\begin{proposition}
  There is no $n$-dimensional rational projective plane for $128 < n < 256$.
\end{proposition}
\begin{proof}
  By Proposition~\ref{proposition:2-adic}, we have $n \in \{136, 144,
  160, 192\}$.  But $n \neq 136$ and $n \neq 160$ by the witnesses in
  Table~\ref{table:witnesses}.  We can rule out a 144-dimensional
  rational projective plane by applying
  Lemma~\ref{lemma:rule-out-using-two-primes} with
  \begin{align*}
    p &= 1872341908760688976794226499636304357567811, \\
    q &= 228479.
  \end{align*}
  And we can rule out a 192-dimensional rational projective plane by choosing
  \begin{align*}
    p &= 4155593423131, \\
    q &= 191.
  \end{align*}
\end{proof}

We can also direct our attention to even higher dimensions.  When is
there a $2^{k+3}$-dimensional rational projective plane?  The fact that
divisors of $2^{2^k - 1} - 1$ are rarely (never?) $5 \bmod 8$ tells us
not to look for such primes among the divisors of the Mersenne factor.
The desiderata of Lemma~\ref{lemma:rule-out} are satisfied by finding
a prime $p$ so that
\begin{itemize}
\item $p \equiv 5 \pmod 8$,
\item $p > 4 \cdot 2^k$,
\item $p$ divides the numerator of $B_{4 \cdot 2^k}$,
\item $p$ does not divide the numerator of $B_{2 \cdot 2^k}$.
\end{itemize}

For example, the prime $p = 502261$ is $5 \bmod 8$ and divides the
numerator of $B_{4 \cdot 2^{11}}$ but not $B_{2 \cdot 2^{11}}$, which
rules out a rational projective plane in dimension $2^{14} = 4096$;
similarly, the prime $p = 69399493$ is $5 \bmod 8$, divides the
numerator of $B_{4 \cdot 2^21}$ but not $B_{2 \cdot 2^21}$, which
rules out a projective plane in dimension $2^{24}$.  These
calculations are possible due to tables of irregular primes produced
by Joe P.~Buhler and David Harvey \cite{MR2813369}.

\section{Connectedness of the 32 dimensional rational projective planes}
\label{section:connectivity}

There exist 32-dimensional rational projective planes.  In light of
this, we consider the torsion structure of such closed smooth
manifolds. We show that a 32-dimensional rational projective plane
does not admit a Spin structure.
\begin{lemma}\label{lemma:32spin}
There does not exist a simply-connected closed Spin manifold $M$ in dimension 32 such that
\[
H^*(M;\Q)\cong
  \begin{cases}
    \Q & \mbox{if $\ast =0, 16, 32$, and} \\
    0  & \mbox{otherwise}.
  \end{cases}
\]
\end{lemma}
The key application of this theorem is to show that a 32-dimensional
rational projective plane cannot be very highly (integrally) connected.
\begin{repmaintheorem}{maintheorem:not-two-connected}
Any 32-dimensional rational projective plane is not 2-connected.
\end{repmaintheorem}
\begin{proof}
  Suppose $M$ is a 32-dimensional rational projective plane. By Lemma~\ref{lemma:32spin}, $M$ is not Spin.  But the
  Stiefel-Whitney class $w_2(M)$ is the obstruction to a Spin
  structure, so $w_2(M) \neq 0$, and so $H^2(M,\Z_2) \neq 0$, and so
  by universal coefficients, $H^2(M,\Z) \neq 0$.
\end{proof}

\begin{proof}[Proof of Lemma~\ref{lemma:32spin}]
If there exists such a simply-connected closed smooth manifold $M^{32}$, the Hizebruch signature theorem implies that  
\[
s_{4,4}\langle p_{4}^2(\tau_M),[M]\rangle+s_{8}\langle p_{8}(\tau_M),[M]\rangle=\pm 1.
\]
As in Section~\ref{section:rational-planes}, if $\alpha\in H^{16}(M;\Z)$ is the generator such that $\langle\alpha\smile \alpha, [M]\rangle=\pm 1$, then $p_{4}(\tau_M)=x\alpha$ and $p_8(\tau_M)=y\alpha^2$ for integers $x$ and $y$, then we may write $\langle p_{4}^2(\tau_M),[M]\rangle=x^2$ and $\langle p_{8}(\tau_M),[M]\rangle=y$  to have the diophantine equation
\begin{equation}
\label{eq32}
s_{4,4}x^2+s_{8}y=\pm 1
\end{equation}
Moreover, the integers $x^2$ and $y$ must be the Pontragin numbers of
a genuine Spin manifold of dimension 32. The following theorem of
Stong characterized the congruence relations that determines the set
of all possible Pontryagin numbers of a Spin manifold.

\begin{theorem}[\cite{MR192516}] \label{hs} For $8k$-dimensional closed Spin manifolds, the stable tangent bundle $\tau_N:N\to BSpin$ induces a homomorphism
$$\tau: \Omega^{Spin}_{8k}/\mathrm{tor}\rightarrow H_{8k}(BSpin;\Q).$$
The image of the homomorphism $\tau$ is a lattice consisting exactly the elements
$x\in H_{8k}(BSpin;\Q)$ such that
\begin{equation}\label{stong}
    \langle \Z[e_1(\gamma), e_2(\gamma),\cdots ]\cdot \hat{A}(p_i(\gamma)), x\rangle\in\Z.
\end{equation}
\end{theorem}
We explain the notations in the theorem. The $KO$-theoretic Pontryagin
character $e_i(\gamma)$ is defined as follows: the total Pontryagin
class of the universal vector bundle $\gamma$ over $BSpin$ can be
formally expressed as $p(\gamma)=\Pi(1+x_j^2)$ by the splitting
principle. The class $e_i(\gamma)\in H^*(BSpin;\Q)$ is the $i$-th
elementary symmetric polynomial of the variables $e^{x_j}+e^{-x_j}-2$,
i.e.,
\[
e_i(\gamma)=\sigma_i\left(e^{x_1}+e^{-x_1}-2, e^{x_2}+e^{-x_2}-2, \cdots\right).
\]
Here, $\hat{A}(p_i(\gamma))$ denotes the total $\hat{A}$-polynomial of
the Pontryagin classes $p_i(\gamma)$'s.

Since the Pontryagin class $p_i(\gamma)$ is exactly the $i$-th
elementary symmetric polynomial in the variables $x_j^2$'s, each class
$e_i(\gamma)$, which can be expanded as a symmetric polynomials of the
variables $x_j^2$'s, can be written as a polynomial in the
$p_i(\gamma)$'s. Therefore Equation~\eqref{stong} can be expressed as
a set of congruence relations on the Pontryagin numbers $\langle
p_I(\gamma), x \rangle$ over the partitions $I$ of $2k$.

Since for any closed Spin manifold $N$, $\langle
p_I(\tau_N),[N]\rangle=\langle p_I(\gamma),\tau_*[N]\rangle=\langle
p_I(\gamma), x\rangle$, the relations on $\langle p_I(\gamma),
x\rangle$ in \eqref{stong} determine a set of integrality conditions
on the Pontryagin numbers $\langle p_I(\tau_N),[N]\rangle$. Therefore
Theorem~\ref{hs} implies the following Lemma, which characterizes all
the possible Pontryagin numbers of a closed Spin manifold.

\begin{corollary}\label{spin manifold}
If $M$ is an $8k$-dimensional closed Spin manifold, then 
\begin{equation}\label{congruence spin}
\langle \Z[e_1(\tau_M), e_2(\tau_M)\cdots ]\cdot  \hat{A}(p_i(\tau_M)), \ [M]\rangle\in\Z,
\end{equation}
where each class $e_i (\tau_M)\in H^*(M;\Q)$ is the pull back of $e_i(\gamma)\in H^*(BSpin;\Q)$ by the stable tangent bundle $\tau_N:N\to BSpin$. \\
\end{corollary}

In our case, a rational projective plane $M^{32}$ has only nonzero
Pontryagin classes $p_4\in H^{16}(M;\Q)$ and $p_8\in
H^{32}(M;\Q)$. For dimension $32$, the expressions of the $e_i$
classes in terms of $p_4$ and $p_8$ was calculated in
\cite{MR3158765}.  This data is recalled in Table~\ref{table:euler-class-expressions}.

\begin{table}
\caption{Expression for $e_i$ in terms of $p_4$ and $p_8$ from \cite{MR3158765}.}
\label{table:euler-class-expressions}

\begin{align*}
  e_1&=-\frac{1}{5040}p_4+\frac{1}{2615348736000}p_4^2-\frac{1}{1307674368000}p_8\\
  e_2&=\frac{1}{40}p_4+\frac{3119}{435891456000}p_4^2+\frac{5461}{217945728000}p_8,\\
  e_1e_1&=\frac{1}{25401600}p_4^2 \\
  e_3&=-\frac{1}{3}p_4 + \frac{19}{39916800}p_4^2-\frac{31}{2851200}p_8,  \quad     e_1e_2=-\frac{1}{201600}p_4^2,  \quad  e_1^3=0\\
  e_4&=p_4 + \frac{1}{1209600}p_4^2 + \frac{457}{604800}p_8,  \quad   e_1e_3=\frac{1}{15120}p_4^2,  \quad    e_2e_2=\frac{1}{1600} p_4^2\\
  e_5&=-\frac{43}{2520}p_8, \quad  e_1e_4=-\frac{1}{5040}p_4^2,    \quad  e_2e_3=-\frac{1}{120}p_4^2, \\
  e_6&=\frac{29}{180}p_8,  \quad   e_2e_4=\frac{1}{40}p_4^2,  \quad  e_3e_3=\frac{1}{9}p_4^2, \quad  e_1e_5=0\\
  e_7&=-\frac{2}{3}p_8,  \quad  e_3e_4=-\frac{1}{3}p_4^2, \quad   e_2e_5=0, \quad  e_1e_6=0 \\
  e_8&=p_8,  \quad  e_4e_4=p_4^2,  \quad  e_3e_5=0,  \quad  e_2e_6=0,  \quad  e_1e_7=0   
\end{align*}
\end{table}

We are now in a position to calculate the $\hat{A}$-genus. For each partition $I$ of $k$, let $a_I$ denote the coefficient of the $I$-th Pontryagin class $p_I$ in the $k$-th $ \hat{A}$-polynomial, then similar to the derivation of the coefficients in the $\LP$-genus through Equation~\eqref{sk} and Equation~\eqref{sk2}, we have corresponding formulas
\begin{align*}\label{ak}
a_k&=-\df{|B_{2k}|}{2(2k)!}, \\
a_{k, k}&=\frac{1}{2}(a_k^2-a_{2k}).
\end{align*}
The total $\hat{A}$-polynomial up to dimension 32 in terms of $p_4$ and $p_8$ is
\begin{equation}\label{A}
\hat{A}=1-\df{1}{2419200}p_4+\df{14527}{85364982743040000} p_4^2- \df{3617}{21341245685760000}p_8.
\end{equation}

\begin{table}
\caption{Basis of $\Z[e_1, e_2\cdots ]\cdot  \hat{A}$ in dimension~32.}
\label{table:basis-for-ahat}
\begin{align*}
  1\cdot \hat{A}&= \frac{14527}{85364982743040000} p_4^2- \frac{3617}{21341245685760000}p_8 \\                   
  e_1\cdot \hat{A}&= \frac{431}{5230697472000} {p_4}^2-\frac{1}{1307674368000}{p_8} \\                     
  e_2\cdot \hat{A}&=-\frac{2771}{871782912000}p_4^2+\frac{5461}{217945728000}p_8
  ,\ \ \ \ \ \ \ e_1e_1\cdot \hat{A}=\frac{1}{25401600}p_4^2 \\                    
  e_3\cdot \hat{A}&=\frac{7}{11404800}p_4^2-\frac{31}{2851200}p_8,\quad \quad \ e_1e_2\cdot \hat{A}=-\frac{1}{201600}p_4^2 \\                      
  e_4\cdot \hat{A}&=\frac{1}{2419200}p_4^2+\frac{457}{604800}p_8, \quad e_1e_3\cdot \hat{A}=\frac{1}{15120}p_4^2, \quad e_2e_2\cdot \hat{A}=\frac{1}
  {1600}p_4^2\\                   
  e_5\cdot \hat{A}&=-\frac{43}{2520}p_8, \quad e_1e_4\cdot \hat{A}=-\frac{1}{5040}p_4^2, \quad e_2e_3\cdot \hat{A}=-\frac{1}{120}p_4^2 \\                     
  e_6\cdot \hat{A}&= \frac{29}{180}p_8,\quad  e_2e_4\cdot \hat{A}=\frac{1}{40}p_4^2, \quad e_3e_3\cdot \hat{A}=\frac{1}{9}p_4^2 \\                        
  e_7\cdot \hat{A}&=-\frac{2}{3}p_8, \quad \quad e_3e_4\cdot \hat{A}=-\frac{1}{3}p_4^2 \\
  e_8\cdot \hat{A}&= p_8, \quad \quad e_4e_4\cdot \hat{A}= p_4^2.\nonumber
\end{align*}
\end{table}

The basis of $\Z[e_1, e_2\cdots ]\cdot  \hat{A}$ in dimension~32 is displayed in Table~\ref{table:basis-for-ahat}.
Corollary \ref{spin manifold} says each of the above basis class should satisfy $\langle -, [M]\rangle\in \Z$. Let $\langle p_4^2, [M]\rangle=x^2$ and $\langle p_8, [M]\rangle=y$, the integrality conditions can then be simplified to the following congruence relations
\begin{align*}
85364982743040000 &\divides 14527 x^2-14468 y\\
5230697472000 &\divides 431 x^2-4 y\\
871782912000 &\divides -2771 x^2 + 21844 y\\
11404800 &\divides 7 x^2 - 124 y\\
2419200 &\divides x^2 + 1828 y\\
25401600 &\divides x^2\\
2520 &\divides y
  \end{align*}
Recall that the signature equation says
\[
  -444721 x^2+118518239 y=\pm162820783125.
\]
We will show that if the above integrality conditions are satisfied,
the signature equation does not have any solution. Since $2520
\divides y$, the signature equation is equivalent to the following
quadratic residue problem\begin{align}\label{qr}
  -444721x^2 &\equiv\pm162820783125\pmod{(118518239)\cdot(2520)}\\
  x^2 &\equiv\pm(-444721)^{-1}\cdot162820783125\pmod{298665962280} \nonumber\\
  x^2 &\equiv\pm11010868155\pmod{298665962280}.\nonumber
\end{align}
The modulus has the prime factorization $298665962280=2^3\cdot
3^2\cdot 5^1\cdot 7^2\cdot31\cdot151\cdot3617$, since
$11010868155\equiv 3\pmod{8}$ and $-11010868155\equiv 5\pmod{8}$, both
$11010868155$ and $-11010868155$ are quadratic nonresidue mod
$2$. Therefore, Equation~\eqref{qr} has no solution. This shows that
there does not exist a simply-connected closed Spin manifold $M^{32}$
such that $H^*(M;\Q)=\Q[x]/\langle x^3\rangle$. We conclude that none
of the 32-dimensional rational projective planes admit a Spin
structure.
\end{proof}

\section{Rational projective spaces}
\label{section:octonionic-examples}
 
The construction of the octonionic projective plane $\OP^2$ does not
generalize to any projective space $\OP^m$ for $m>2$. If
$H^*(X;\Z)=\Z[x]/\langle x^{m+1}\rangle$ with $m>2$, then $|x| \in \{ 2, 4
\}$ (\cite{MR1867354}). We study the existence of ``rational octonionic
spaces.''  In other words, we ask whether there exist a closed smooth
manifold $M^{8k}$ such that $H^*(M;\Q)=\Q[x]/\langle x^{m+1}\rangle,
|x|=8$ for $m>2$.  We apply rational surgery to show the existence of
rational $\OP^{m}$ for $m$ odd.

Our main technical tool is the following result of Barge and Sullivan.

\begin{theorem}[Barge \cite{MR440574}, Sullivan \cite{MR646078}]\label{bs}
  Let $X$ be an $n=4k$-dimensional simply-connected,
  $\mathbb{Q}$-local, $\mathbb{Q}$-Poincar\'{e} complex, where
  $k\neq1$. There exists a simply-connected $4k$-dimensional, closed,
  smooth manifold $M$, and a $\mathbb{Q}$-homotopy equivalence $f:
  M\to X$ if and only if there exist cohomology classes $p_i \in
  H^{4i}(X;\mathbb{Q})$ for $i=1, \ldots, k$, and a fundamental class
  $\mu \in H_{4k}(X;\mathbb{Q})\cong\Q$ such that

\textup{(i)} the pairing of the $k$th $L$-polynomial of $p_i$'s and $\mu$ is
  equal to the signature of $X$, i.e., $\langle L_k(p_1,\ldots, p_k),
  \mu\rangle=\sigma(X)$;

\textup{(ii)} the intersection form $\lambda: H^{2k}(X;\Q)\times
  H^{2k}(X;\Q)\rightarrow \Q$ defined as
  $\langle\cdot\smile\cdot,\mu\rangle$ is isomorphic to a direct sum
  of copies of $\langle1\rangle$'s and $\langle-1\rangle$'s; and

\textup{(iii)} the pairings $\langle p_I, \mu\rangle=\langle p_{i_1}\cdots
  p_{i_r}, \mu\rangle$ over all the partitions $I=(i_1,\ldots,i_r)$ of
  $k$ form a set of Pontryagin numbers of a genuine closed smooth
  manifold, i.e., there exists a $4k$-dimensional closed smooth
  manifold $N$ such that $$\langle p_I(\tau_N),[N]\rangle=\langle p_I,
  \mu\rangle$$ for all partitions $I$ of $k$.

If the choice of $p_i$'s and $\mu$ satisfies all the conditions above,
surgery theory will construct a $\Q$-homotopy equivalence $f: M\to X$
such that $f_*[M]=\mu$ and $f^*(p_i)=p_i(\tau_M)$, where $p_i(\tau_M)$
is the $i$-th Pontryagin class of the tangent bundle of $M$. As a
consequence, the Pontryagin numbers $p_I[M]=\langle p_I, \mu\rangle$
for all partitions $I$ of $k$.
\end{theorem}

\begin{repmaintheorem}{maintheorem:rational-op3}
 If $m>2$ is odd, there exists a closed smooth $8m$-dimensional manifold $M$ with rational cohomology ring
  $$
  H^*(M;\Q) \cong \Q[x]/\langle x^{m+1}\rangle \quad |x| = 8.
  $$
\end{repmaintheorem}
\begin{proof}
  If a rational Poincare duality algebra $\A$ is intrinsically formal,
  it contains a unique rational homotopy type (\cite{MR645331}), i.e., for
  any two simply-connected spaces $X$ and $Y$ such that $H^*(X;
  \Q)\cong H^*(Y; \Q)\cong \A$, $X$ and $Y$ are rational homotopy
  equivalent to each other. Any truncated rational polynomial algebra
  is intrinsically formal (\cite{MR645331}), so $\A=\Q[x]/\langle
  x^4\rangle, |x| = 8$ is intrinsically formal. Similar to the
  approach to study the existence of rational projective planes in
  \cite{MR3158765}, we construct a $\Q$-local space $X$ such that
  $H^*(X;\Z)\cong H^*(X;\Q)\cong \A$, then apply the
  rational surgery realization Theorem~\ref{bs} to determine if there
  exists a manifold $M$ which is rational homotopy equivalent to $X$,
  thus $H^*(M; \Q)\cong\A$. Given a rational homotopy type, the
  theorem provides the sufficient and necessary condition for the
  existence of a simply-connected closed smooth manifold within the
  rational homotopy type.

  Now, consider Theorem~\ref{bs}.  In our case,
  $H^*(X;\Q)\cong\Q[x]/\langle x^{m+1}\rangle, |x|=8$. Note that when $m$ is odd, $8$ does not divide $4m$, therefore the middle dimensional cohomology group
  $H^{4m}(X;\Q)=0$, so there is no obstruction from condition (ii). Moreover, the signature $\sigma(X)=0$. Then in the rational
  surgery realization Theorem~\ref{bs}, the choice that each cohomology class 
  $p_i=0$ with any fundamental class $\mu$ would satisfy
  condition (i),(ii) and (iii). Such a realizing manifold has all the
  Pontraygin numbers vanish.
\end{proof}

\begin{remark} The situation of determining the existence of rational $\OP^m$ for $m$ even turns out to be more complicated. In the 
case of rational $\OP^4$, $A=\Q[x]/\langle x^5\rangle$. Let $p_1=ax$, $p_2=bx^2$,
$p_3=cx^3$, $p_4=dx^4$, then $p_{1,1,1,1}=a^4x^4$, $p_{1,1,2}=a^2b$,
$p_{1,3}=ac$, $p_{2,2}=b^2$, $p_4=d$. The signature equation requires
\[
s_{2,2,2,2}a^4+s_{2,2,4}a^2b+s_{2,6}ac+s_{4,4}b^2+s_8d=\pm 1
\]
Setting $a = 0$ and $b = 20688922800$ and $c = 0$ and $d =
1606120797592276875$ yields a solution, but it is not yet obvious
whether or not the resulting Pontryagin numbers are the genuine Pontryagin
numbers of a smooth manifold.  More computation needs to be done.
\end{remark}

\section{Rational homotopy type of Milnor $E_8$ manifolds}
\label{section:e8-manifolds}

Along the line of realizing a rational homotopy type by smooth
manifolds, we consider the rational homotopy type of the Milnor $E_8$
manifolds. The $E_8$ manifolds $M^{4k}$ are $(2k-1)$-connected closed
topological manifolds constructed by plumbing together disc bundles
over the spheres according to the $E_8$ diagram.  Such a manifold does
not admit any smooth structure.  Nevertheless, one can ask whether the
$E_8$ manifold has the rational homotopy type of a smooth manifold.  We
prove two nonexistence and one existence results.

\begin{proposition}
  \label{proposition:E8-8k-realization}
  The $E_8$ manifold $M^{4k}$ does not have the rational homotopy type
  of a smooth manifold when $k>1$ is odd.
\end{proposition} 

\begin{proof}
  Suppose $k > 1$ is odd and that there exists a closed smooth
  manifold $N^{4k}$ which is rational homotopy equivalent to $M$, so
  $H^*(N,\Q)\cong H^*(M;\Q)$.  Note that $p_k[N]$ is the only nonzero
  Pontryagin number of $N$, so the signature equation requires
  \[
  \langle\LP(p(\tau_N)), [N]\rangle=s_kp_k[N]=\sigma(N)=\sigma(M)=8.
  \]
  Moreover, since $N$ is a smooth manifold, $p_k[N]$ is an integer.

  But, as we will see, there is no integer $x$ so that $s_kx=8$ when $k>1$ is
  odd. Suppose there is such an integer $x$, then
 
  \[
  s_kx=\df{2^{2k}(2^{2k-1}-1)\numer(B_{2k})}{(2k)!\denom(B_{2k})}x=\df{2^{2k}(2^{2k-1}-1)\numer(\frac{B_{2k}}{2k})}{(2k-1)!\denom(\frac{B_{2k}}{2k})}x=8.
  \]
  It is known that $\numer(\frac{B_{2k}}{2k})=1$ only for $2k =
  2,4,6,8,10,14$, otherwise $\numer(\frac{B_{2k}}{2k})$ is a product of powers of irregular
  primes $p$ such that $p>2k+1$. 
 
   Suppose $k \not\in\{ 3, 5, 7 \}$; in this case, let $p$ be an
  irregular prime that divides $\numer(\frac{B_{2k}}{2k})$, then the signature equation implies $p$ divides $8\cdot(2k-1)!\cdot \denom(\frac{B_{2k}}{2k})$. But since the odd irregular prime $p>2k+1$, $p$ does not divide $8\cdot (2k-1)!\cdot
  \denom(\frac{B_{2k}}{2k})$. Therefore $s_kx=8$ has no integer solution $x$.

  The other cases can be handled individually.  When $k \in \{3, 5,
  7\}$, the signature equations are
  \[
  s_3x=\frac{62}{945}x=8, \quad s_5x=\df{146}{13365}x=8, \quad s_7x=\frac{32764}{18243225}x=8
  \]
  respectively, and each equation has no integer solution.
\end{proof}

Similar to the proof of Theorem~\ref{maintheorem:two-nonzero-bits}, we use the 2-adic
valuation of $s_{k,k}$ and $s_{2k}$ to show the following nonexistence
result.

  \begin{proposition}\label{proposition:E8-two-adic-nonrealization}
    The $E_8$ manifold $M^{8k}$ does not have the rational homotopy
    type of a smooth manifold when the binary expansion of $k$ has
    more than 5 nonzero bits.
  \end{proposition}

  \begin{proof}
    Suppose there exists a closed smooth manifold $N^{8k}$ which is rational homotopy equivalent to $M$.  In this case, the signature equation says
    \[
    s_{k,k}\langle p_{k,k}, \mu \rangle + s_{2k} \langle p_{2k}, \mu\rangle = 8.
    \]
    where $\langle p_{k,k}, \mu\rangle$ and $\langle p_{2k}, \mu\rangle$ are
    integers. As in Section~\ref{section:rational-planes}, define
    $\wt(k)$ to be the number of nonzero bits in the binary expansion
    of $k$. By Proposition~\ref{proposition:2-adic}, the numerators of
    both $s_{k,k}$ and $s_{k}$ are divisible by $2^{\wt(k)}-2$. When
    $\wt(k)>5$, the left hand side of the signature equation is
    divisible by $16$, and therefore the equation has no integer
    solution.
\end{proof}
A specific instance of Proposition~\ref{proposition:E8-two-adic-nonrealization} is
in dimension $504$; note that
\[
504 = 2^3 + 2^4 + 2^5 + 2^6 + 2^7 + 2^8,
\]
and so $\wt(504) = 6$.  The Milnor $E_8$ manifold $M^{504}$ is a
topological manifold which does not have the rational homotopy type of
a smooth manifold.

\begin{proposition}
  \label{proposition:E8-realization}
  The $8$-dimensional $E_8$ manifold $M^{8}$ has the rational homotopy
  type of a smooth manifold.
\end{proposition} 

\begin{proof} 
  Following our approach to the realization of rational projective
  planes and rational $OP^4$, we again apply the surgery realization
  Theorem~\ref{bs}. Let $X$ be the $\Q$-localization of $M^{4k}$, so
  $H^*(X;\Z)\cong H^*(M^8;\Q)$. Let $a_i\in H^4(X;\Q)=\Q^8$ be a
  generator of the $i$-th summand of $\Q^8$, and let $\mu\in
  H_8(X;\Q)$ be a fundamental class, such that the intersection form
  on $H^4(X;\Q)$ with respect to the basis $a_i$'s and $\mu$ is
  exactly the $E_8$ form. We seek a choice of $p_1, p_2\in H^*(X;\Q)$
  such that the signature formula
  \begin{equation}\label{dim8}
    s_{1,1}\langle p_1^2,\mu\rangle+s_{2}\langle p_{2},\mu\rangle=8
  \end{equation}
  is satisfied. Note that $p_1^2[\CP^4] = 25$, and $p_2[\CP^4] = 10$.
  The signature formula says $s_{1,1}p_1^2[\CP^4]+s_2p_{2}[\CP^4]
  =1$. Then we let $p_1=10 a_1$, $p_2=40a_1^2$ so that
  \begin{align*}
    \langle p_{1,1},\mu\rangle=100\langle a_1^2,\mu\rangle=200&=8p_1^2[\CP^4],\\
    \langle p_{2},\mu\rangle=40\langle a_1^2,\mu\rangle=80&=8p_{2}[\CP^4].
  \end{align*}
  Then the signature formula \eqref{dim8} is satisfied.  Also, $\langle
  p_{1,1},\mu\rangle$ and $\langle p_{2},\mu\rangle$ are the Pontragin
  numbers of a genuine closed smooth manifold. Having met all the
  conditions of the surgery realization Theorem~\ref{bs}, we may
  conclude that $M^{8}$ has the rational homotopy of a smooth manifold.
\end{proof}

Propositions~\ref{proposition:E8-8k-realization},
\ref{proposition:E8-two-adic-nonrealization}, and
\ref{proposition:E8-realization} combine to yield
Theorem~\ref{maintheorem:e8-manifolds}.
 
\appendix
\section{Computing the $L$-polynomial}
\label{appendix:recursive-forumula}

In general, our approach of studying the realization of a rational
cohomology ring by smooth manifold requires finding the coefficients
of the $k$-th $\LP$ polynomial. This is harder than it may seem at
first.  A na\"\i ve approach for this calculation is to directly
express the homogenous part of degree $k$ in
\[
1+\LP_1(p_1)+\LP_2(p_1, p_2)+\ldots+\LP_k(p_1, \ldots, p_k)=f(t_1)f(t_2)\cdots f(t_k),
\]
by the elementary symmetric polynomials $p_i=\sigma_i(t_i)$, 
where the generating function 
\[
f(t)= \frac{\sqrt{t}}{\tanh \sqrt{t}} = \sum_{k=0}^\infty \frac{2^{2k} \,
    B_{2k} \, t^k}{(2k)!} = 1 + \frac{t}{3} - \frac{t^2}{45} + \cdots.
\]
This approach involves expanding the product of power series of large
degree; it is not nearly efficient enough for our desired
applications.  Here we give an recursive algorithm for calculating the
$\LP$ polynomial.

For each partition $I=i_1, \ldots, i_r$ of $k$, let $s_I$ denote the
coefficient of the $I$-th Pontryagin class $p_I=p_{i_1}\cdots p_{i_r}$
in the $k$-th $\LP$-polynomial, i.e., we express
\[
\LP_k(p_1, p_2, \ldots, p_k)=\sum_{|I|=k}s_Ip_I
\]

From \cite{MR440554}, $s_k$ can be calculated by the formula
\begin{equation}\label{sk}
  s_k=\displaystyle\frac{2^{2k}(2^{2k-1}-1)|B_{2k}|}{(2k)!}.
\end{equation}
And in \cite{MR238332}, it was derived that
\begin{equation}\label{sk2}
  s_{k, k}=\frac{1}{2}(s_k^2-s_{2k})
\end{equation}
This is the first case of a recursive formula.  In general, for any
partition $I$ of $k$, one can calculate $s_I$ in terms of the $s_J$'s,
where $I$ is a refinement of $J$.
\begin{proposition}\label{s}
Let $I=\underbrace{i_1,\ldots,i_1}_{\mu(i_1)}, \cdots, \underbrace{i_n,\ldots i_n}_{\mu(i_n)}$ be a partition of $k$, then 
\begin{equation}\label{sI}
  s_I=\df{1}{\mu(i_1)!\cdots \mu(i_n)!}\left(s_{i_1}^{\mu(i_1)}\cdots s_{i_n}^{\mu(i_n)}-\ds_{|J|=k,J>I}a_{J}\, \mu(j_1)!\cdots \mu(j_m)!\, s_J\right),
\end{equation}
where $J>I$ means that $J$ is any partition of $k$ such that $I$ is a proper refinement
of $J$. For each
$J=\underbrace{j_1,\ldots,j_1}_{\mu(j_1)}, \cdots,
\underbrace{j_m,\ldots j_m}_{\mu(j_m)}$, the coefficient $a_J$ counts
the number of partitions $P=\{A_1, \ldots, A_m\}$ of the multiset
$\{I\}$ such that $J=\ds_{a\in A_1}a, \ldots, \ds_{a\in A_m}a$.
\end{proposition}

\begin{proof}
   For any partition $I=\underbrace{i_1,\ldots,i_1}_{\mu(i_1)}, \cdots,
  \underbrace{i_n,\ldots i_n}_{\mu(i_n)}$ of $|I|=k$, consider the
  corresponding product of spheres $X=\dpr_{l=1}^{\mu(i_1)}
  S^{4i_1}_{(l)}\times\cdots \times\dpr_{l=1}^{\mu(i_n)}
  S^{4i_n}_{(l)}$. Let $\xi_{i}^{(l)}$ be any real vector bundle over
  the $l$-th $4i$-dimensional sphere $S^{4i}_{(l)}$ in the product.
  We denote the total Pontraygin class
  $p(\xi_{i}^{(l)})=1+p_{i}^{(l)}$, where $p_{i}^{(l)}\in
  H^{4i}(S^{4i}_{(l)};\Z)$. Consider the product bundle
  $\eta=\dpr_{l=1}^{\mu(i_1)} \xi_{i_1}^{(l)}\times\cdots
  \times\dpr_{l=1}^{\mu(i_n)} \xi_{i_n}^{(l)}$, the total Pontryagin
  class $p(\eta)=\dpr_{l=1}^{\mu(i_1)} (1+p_{i_1}^{(l)})\times\cdots
  \times\dpr_{l=1}^{\mu(i_n)} (1+p_{i_n}^{(l)})$. In the
  $\LP$-polynomial $\LP_{|I|}(p(\eta))$, the coefficient of
  $\dpr_{l=1}^{\mu(i_1)} p_{i_1}^{(l)}\times\cdots\times
  \dpr_{l=1}^{\mu(i_n)} p_{i_n}^{(l)}$ is $\ds_{J\geq I}a_{J}\,
  \mu(j_1)!\cdots \mu(j_m)!\, s_J$, where $J$ is any partition of
  $|I|$ such that $I$ is a refinement of $J$ ($J\geq I$).

On the other hand, by the multiplicativity of the $\LP$-polynomials,
\begin{align*}
\LP(p(\eta))&=\dpr_{l=1}^{\mu(i_1)} \LP(p(\xi_{i_1}^{(l)})\times\cdots \times\dpr_{l=1}^{\mu(i_n)} \LP(p(\xi_{i_n}^{(l)}) \\
&=\dpr_{l=1}^{\mu(i_1)} (\cdots+s_{i_1}p_{i_1}^{(l)}+\cdots)\times\cdots \times\dpr_{l=1}^{\mu(i_n)}  (\cdots+s_{i_n}p_{i_n}^{(l)}+\cdots) \\
&= \cdots+ s_{i_1}^{\mu(i_1)}\cdots s_{i_n}^{\mu(i_n)}\left(\dpr_{l=1}^{\mu(i_1)} p_{i_1}^{(l)}\times\cdots\times \dpr_{l=1}^{\mu(i_n)} p_{i_n}^{(l)}\right)+\cdots
\end{align*}
Having obtained two expressions for the coefficient of $\dpr_{l=1}^{\mu(i_1)} p_{i_1}^{(l)}\times\cdots\times \dpr_{l=1}^{\mu(i_n)} p_{i_n}^{(l)}$ in $\LP(p(\eta))$, we may equate them to get the equation
\begin{align*}
s_{i_1}^{\mu(i_1)}\cdots s_{i_n}^{\mu(i_n)}&=\ds_{|J|=k, J\geq I}a_{J}\, \mu(j_1)!\cdots \mu(j_m)!\, s_J \\
&=\mu(i_1)!\cdots \mu(i_n)!\, s_I+\ds_{|J|=k, J>I}a_{J}\, \mu(j_1)!\cdots \mu(j_m)!\, s_J,
\end{align*}
from which we solved
\[
s_I=\df{1}{\mu(i_1)!\cdots \mu(i_n)!}\left(s_{i_1}^{\mu(i_1)}\cdots s_{i_n}^{\mu(i_n)}-\ds_{|J|=k, J>I}a_{J}\, \mu(j_1)!\cdots \mu(j_m)!\, s_J\right)
\]
\end{proof} 

\begin{example}
\label{example:computing-s1122}
Consider the case of finding $s_{1,1,2,2}$. The corresponding product
of sphere is $X=S^4_{(1)}\times S^4_{(2)}\times S^8_{(1)}\times
S^8_{(2)}$, over which any vector bundle
$\eta=\xi_1^{(1)}\times\xi_1^{(2)}\times\xi_2^{(1)}\times\xi_2^{(2)}$
has the total Pontryagin class
$p(\eta)=(1+p_1^{(1)})\times(1+p_1^{(2)})\times(1+p_2^{(1)})\times(1+p_2^{(2)})$. We
want to find the coefficient of $p_1^{(1)}\times p_1^{(2)}\times
p_2^{(1)}\times p_2^{(2)}$ in the $\LP$-polynomial
$\LP_{6}(p(\eta))=\sum_Js_{J}p_{J}(\eta)$. Firstly note that
$p_1^{(1)}\times p_1^{(2)}\times p_2^{(1)}\times p_2^{(2)}$ is only
contained in a $p_J(\eta)$ where $J\geq I$, i.e., partition $J$ such that $I=1,1, 2, 2$ is a
refinement of $J$.

For starters, consider $J=(2,4)$,
\begin{align*}
p_{2,4}(\eta)&=p_2(\eta)p_4(\eta) \\
&=(p_1^{(1)}\times p_1^{(2)}\times 1\times 1 + 1\times 1\times p_2^{(1)}\times 1+1\times 1\times 1\times p_2^{(2)}) \\
&\quad (p_1^{(1)}\times p_1^{(2)}\times p_2^{(1)}\times 1 + p_1^{(1)}\times p_1^{(2)}\times 1\times p_2^{(2)}+1\times 1\times p_2^{(1)}\times p_2^{(2)})\\
&=(3)(p_1^{(1)}\times p_1^{(2)}\times p_2^{(1)}\times p_2^{(2)})+\mbox{ other terms, }
\end{align*}
where the coefficient $a_J=3$ counts the 3 partitions 
\[
  \{\{1_1,1_2\}, \{2_1,2_2\}\},\  \{\{2_1\},\{1_1,1_2,2_2\}\}, \ \{\{2_2\}, \{1_1, 1_2, 2_1\}\}
\]
of the multiset  $\{I\}=\{1_1,1_2,2_1,2_2\}$. These are the only partitions $\{A_1, A_2\}$ corresponding to $J=(2,4)$ in the sense that $\sum_{a\in A_1}a=2$, and  $\sum_{a\in A_1}a=4$.\\

Next, consider $J=(2,2,2)$,
\begin{align*}
p_{2,2,2}(\eta)&=p_2(\eta)p_2(\eta)p_2(\eta)\\
&=(p_1^{(1)}\times p_1^{(2)}\times 1\times 1 + 1\times 1\times p_2^{(1)}\times 1+1\times 1\times 1\times p_2^{(2)})^3\\
&=(1)(3!)(p_1^{(1)}\times p_1^{(2)}\times p_2^{(1)}\times p_2^{(2)})+\mbox{ other terms, }
\end{align*}
where the coefficient $a_J=1$ counts the single partition $\{A_1, A_2,
A_3\}=\{\{1_1,1_2\}, \{2_1\}, \{2_2\}\}$ of the multiset
$\{I\}=\{1_1,1_2,2_1,2_2\}$ such that $\sum_{a\in A_i}a=2$ for $i=1,
2, 3$. The coefficient $3!$ came from the multiplicity of the parts of
$J$.

We proceed in this fashion, eventually producing the
Table~\ref{table:ajs-for-example} which lists all such partitions $J$
and the corresponding $a_J$.

\begin{table}
\caption{\label{figure:ajs-for-example}Partitions $J$ (and corresponding $a_J$) which refine to $(1,1,2,2)$ for use in Example~\ref{example:computing-s1122}.}
\begin{center}
    \begin{tabular}{  | p{1.5cm} | p{9.8cm} | p{0.5cm} | }    \hline
  $J$ ($J\geq I$)& corresponding partitions of the multiset $\{I\}=\{1_1,1_2,2_1,2_2\}$& $a_J$ \\ 
    \hline
      $(1,1,2,2)$ & $\{\{1_1\},\{1_2\}, \{2_1\}, \{2_2\}\}$& $1$ \\     \hline
    $(2,2,2)$& $\{\{1_1,1_2\}, \{2_1\}, \{2_2\}\}$& $1$\\     \hline
   $(1,2,3)$ & $\{\{1_1\}, \{2_1\}, \{1_2, 2_2\}\}, \{\{1_1\}, \{2_2\},\{1_2, 2_1\}\}, \{\{1_2\}, \{2_1\},\{1_1, 2_2\}\}$, $\{\{1_2\}, \{2_2\},\{1_1, 2_1\}\}$& $4$ \\  \hline
    $(1,1,4)$& $\{\{1_1\}, \{1_2\}, \{2_1,2_2\}\}$ & $1$ \\     \hline
   $(2,4)$& $\{\{1_1,1_2\}, \{2_1,2_2\}\}, \{\{2_1\},\{1_1,1_2,2_2\}\}, \{\{2_2\}, \{1_1, 1_2, 2_1\}\}$ &$3$\\     \hline
   $(3,3)$ & $\{\{1_1,2_1\}, \{1_2,2_2\}\}, \{\{1_1,2_2\}, \{1_2,2_1\}\}$& $2$\\ \hline
   $(1,5)$& $\{\{1_1\},\{1_2, 2_1, 2_2\}\}, \{\{1_2\},\{1_1, 2_1, 2_2\}\}$& $2$ \\     \hline
    $(6)$ & $\{\{1_1,1_2, 2_1, 2_2\}\}$& $1$\\ 
    \hline
  \end{tabular}
\end{center}
\end{table}

The coefficient of $p_1^{(1)}\times p_1^{(2)}\times p_2^{(1)}\times p_2^{(2)}$ in the $\LP$-polynomial $\LP_{6}(p(\eta))=\sum_Js_{J}p_{J}(\eta)$ is then $\ds_{J\geq I}a_{J}\, \mu(j_1)!\cdots \mu(j_m)!\, s_J$. On the other hand, 
\begin{align*}
\LP(p(\eta))&=\LP(p(\xi^{(1)}_1))\times\LP(p(\xi^{(2)}_1))\times\LP(p(\xi^{(1)}_2))\times\LP(p(\xi^{(2)}_2))\\
&=\cdots + (s_1s_1s_2s_2)\,p_1^{(1)}\times p_1^{(2)}\times p_2^{(1)}\times p_2^{(2)}+\cdots.
\end{align*}
Equate the two formulas for the coefficient.  Then we obtain
\begin{align*}s_1^2s_2^2&=\ds_{J\geq I}a_{J}\, \mu(j_1)!\cdots \mu(j_m)!\, s_J \\
&=(1)(2!2!)s_{1,1,2,2}+(1)(3!)s_{2,2,2}+(4)s_{1,2,3}+(1)(2!)s_{1,1,4}+(3)s_{2,4}+(2)(2!)s_{3,3}+(2)s_{1,5}+(1)s_6,
\end{align*}
and from this we conclude
\[
  s_{1,1,2,2}=\df{1}{2!2!}(s_1^2s_2^2-6s_{2,2,2}-4s_{1,2,3}-2s_{1,1,4}-3s_{2,4}-4s_{3,3}-2s_{1,5}-s_6).
\]
\end{example}

We can also find formulas which generalize Equation~\eqref{sk2} for $s_{n,n}$.
\begin{example} 
Here we list some explicit examples of recursive formulas for $s_I$.

\begin{itemize} 
\item If $m\neq n$, then $s_{m,n}=s_ms_n-s_{m+n}.$
\item For any $n$,
\[
s_{n,n,n}=\frac{1}{3!}(s_n^3-3s_{2n,n}-s_{3n}).
\]
\item For any $n$,
\[
s_{n,n,2n}=\frac{1}{2}(s_{2n}s_n^2-2s_{3n,n}-2s_{2n,2n}-s_{4n})
\]
\item If $m\neq n$ and $m\neq 2n$, then
\[
s_{n,n, m}=\frac{1}{2}(s_ms_n^2-s_{m,2n}-2s_{m+n,n}-s_{m+2n}).
\]
\item If $m<n<k$ and $m+n\neq k$, then
\[
s_{m,n,k}=s_ms_ns_k-s_{m+n,k}-s_{m+k,n}-s_{n+k,m}-s_{m+n+k}.
\]
\end{itemize}
\end{example}
Formulas such as these will prove useful in working through more
complicated versions of the spaces considered in
Section~\ref{section:octonionic-examples}.

\bibliographystyle{amsalpha}
\bibliography{references}

\end{document}